\documentclass[12pt]{article}
\usepackage{amsmath}
\usepackage{subfigure}
\usepackage{float}
\usepackage{amsthm}
\usepackage{amssymb, setspace}
\usepackage{color,graphicx, epsfig, geometry, hyperref,fancyhdr}
\usepackage{mathrsfs}
\usepackage{bbm}
\usepackage[T1]{fontenc}
\usepackage{latexsym,amssymb,amsmath,amsfonts}
\usepackage{color}
\usepackage{soul}
\usepackage{caption}
\usepackage{dsfont}
\usepackage{multirow}
\usepackage{rotating}
\newtheorem{theorem}{Theorem}[section]

\newtheorem{lemma}{Lemma}[section]

\newcommand{\N}{\mathbb{N}}
\newcommand{\Z}{\mathbb{Z}}

\newcommand{\R}{\mathbb{R}}
\newcommand{\C}{\mathbb{C}}

\newcommand{\grad}{\nabla}

\newcommand{\ov}{\overline}

\begin{document}
\lhead{}
\rhead{}

\begin{flushleft}
\Large 
\noindent{\bf \Large Approximation of the zero-index transmission eigenvalues with a conductive boundary and parameter estimation}
\end{flushleft}

\vspace{0.2in}

{\bf  \large Isaac Harris}\\
\indent {\small Department of Mathematics, Purdue University, West Lafayette, IN 47907 }\\
\indent {\small Email: \texttt{harri814@purdue.edu}}

\vspace{0.2in}

\begin{abstract}
\noindent In this paper, we present a spectral-Galerkin method to approximate the zero-index transmission eigenvalues with a conductive boundary condition. This is a new eigenvalue problem derived from the scalar inverse scattering problem for an isotropic media with a conductive boundary condition. In our analysis, we will consider the equivalent fourth-order eigenvalue problem where we establish the convergence when the approximation space is the span of finitely many Dirichlet eigenfunctions for the Laplacian. We establish the convergence rate of the spectral approximation by appealing to Weyl's law. Numerical examples for computing the eigenvalues and eigenfunctions for the unit disk and unit square are presented. Lastly, we provide a method for estimating the refractive index assuming the conductivity parameter is either sufficiently large or small but otherwise unknown. 
\end{abstract}

\vspace{0.1in}



\section{Introduction}

In this paper, we consider the numerical approximation of the zero-index transmission eigenvalues that are associated with the scalar scattering problem with a conductive boundary. In general, the transmission eigenvalues can be seen as the wave numbers where the associated far-field {\color{black}operator fails to be} injective. The zero-index transmission eigenvalue problem is derived by mathematically imbedding the scattering object in a background with refractive index equalling zero in the interior of the scatterer. It can be shown that the resulting far-field operator fails to be injective with a dense range at the wave numbers corresponding to these eigenvalues see \cite{zi-te} for the case when the conductivity is zero. The zero-index transmission eigenvalue problem has two main advantages over the classical transmission eigenvalue problem. First, is that they avoid the assumption that the contrast must be either positive or negative definite in the scatterer. Next, is the fact that they are linear  eigenvalue problems. This zero-index eigenvalue problem with a conductive boundary condition was introduced in \cite{two-eig-cbc} and was motivated by the work in \cite{te-cbc,te-cbc2} for the classical transmission eigenvalue problem with a conductive boundary and \cite{zi-te} for the  scattering problem without a conductive boundary.  We are also interested in the inverse spectral problem of estimating the refractive index with little a prior knowledge of the boundary conductivity parameter. There have been manuscripts written on the computation and application of transmission eigenvalue problems to parameter identification such as \cite{LSM-maxwell-book,TE-book,harris-thesis,eig-FEM-book} to name a few. For the classical transmission eigenvalue problem we refer to \cite{spectraltev1,spectraltev2}  for the application of spectral-Galerkin methods to compute the eigenvalues. See for {\color{black}example} \cite{fem-te,fem-te2} for some of the previous work for computing the classical transmission eigenvalues via the finite element method. Recently, the method of fundamental solutions for computing the classical transmission eigenvalues was studied and implemented in \cite{mfs-te}. Due to the monotonicity property of the transmission eigenvalues one can estimate the refractive index from the knowledge of the eigenvalues(see for {\color{black}example} \cite{te-homog,cgtev}). The main contributions of this paper is the convergence analysis with error estimates of the spectral-Galerkin method with the Dirichlet eigenfunctions taken as the basis and the estimation of the refractive index from the zero-index transmission eigenvalues.

The zero-index transmission eigenvalue problem can be written as a {\color{black} fourth-}order eigenvalue problem that depends on the refractive index and conductivity. We now derive the {\color{black} fourth-}order formulation of the eigenvalue problem. To this end, we define the zero-index transmission eigenvalue problem from the scalar isotropic scattering problem as the values $k \in \C\setminus \{ 0 \}$ such that there exists a nontrivial pair $(u,{u_0} ) \in H^1(D) \times H^1(D)$ satisfying the system 
\begin{align}
\Delta u +k^2 n u=0 \quad \text{and} \quad \Delta {u_0} =0  \quad &\textrm{ in } \,  D \label{zi-teprob1} \\
 u-{u_0} =0  \quad  \text{and} \quad {\partial_\nu {u}} = {\partial_\nu {u_0}}+ \eta u_0   \quad &\textrm{ on } \partial D.  \label{zi-teprob2} 
\end{align} 
Here, we assume that $D \subset \R^d$ (for $d=2,3$) is a simply connected open set {\color{black} where the boundary $\partial D$ is either polygonal with no reentrant corners or class $\mathscr{C}^2$} where $\nu$ is the outward unit normal vector. The eigenvalue $k$ corresponds to the wave number for the associated scattering problem. Let the refractive index $n \in L^{\infty} (D)$ and conductivity $\eta \in L^{\infty}(\partial D)$ where we assume that they are uniformly positive definite functions such that there exists positive constants
$$ n_{\text{min}}  \leq n(x) \leq n_{\text{max}}  \quad \text{ a.e.} \, \, \,\,x \in \ov{D} \quad \text{ and } \quad  \eta_{\text{min}} \leq \eta(x)\leq \eta_{\text{max}}\quad \text{ a.e.} \, \, \,\, x \in \partial D.$$
Therefore, we can define the difference of the eigenfunctions $w=u-{u_0}$. It is clear that $w$ satisfies the equation  
$$\Delta w +k^2 n w= -k^2 n {u_0} \quad \textrm{ in } \,  D.$$
Due to standard elliptic regularity results (\cite{evans} page 334 for a $\mathscr{C}^2$ boundary) we have that $w \in H^2(D) \cap H^1_0(D)$. Now, by appealing to the fact that ${u_0}$ is harmonic in $D$ and the boundary condition \eqref{zi-teprob2} we can conclude that $w$ satisfies the homogeneous boundary value problem 
\begin{align}
\Delta \frac{1}{n} \Delta w = -k^2 \Delta w \quad \textrm{ in } \,  D \quad \text{and} \quad  \frac{k^2}{\eta} {\partial_\nu w}= -  \frac{1}{n}\Delta w   \quad \textrm{ on } \,  \partial D. \label{zi-teprob3} 
\end{align} 
In \cite{two-eig-cbc} it is shown that $k \in \C\setminus \{ 0 \}$ is a zero-index transmission eigenvalue problem if and only if there is a nontrivial $w \in H^2(D) \cap H^1_0(D)$ satisfying \eqref{zi-teprob3}. By studying the variational formulation of \eqref{zi-teprob3} it is shown that there exists infinitely many real zero-index transmission eigenvalues. {\color{black} This eigenvalue problem is derived by mathematically embedding scatterer $D$ in a background where the refractive index is equal to zero in the interior of the object. This is done by studying the difference of the far-field operators for the standard scattering problem and the augmented far-field operator for the scattering problem where the refractive index is equal zero in $D$.} In general, it is known that the transmission eigenvalues can be determined from the scattering data. In \cite{armin} it is shown that the classical transmission eigenvalues can be determined from the far-field data. While in \cite{te-cbc2} it is shown that the classical transmission eigenvalues with a conductive boundary can also be recovered from far-field data. This implies that these eigenvalues can be used as a target signature to determine the material properties. 

The rest of the paper is ordered as follows. In the next section we will study the solution operator corresponding to the zero-index transmission eigenvalue problem with a conductive boundary \eqref{zi-teprob3}. We will then consider the approximation of the eigenvalues via a Dirichlet spectral-Galerkin method where the approximation space is taken to be the span of finitely many Dirichlet eigenfunctions for the Laplacian. {\color{black} This method of representing the solution to a PDE by the eigenfunctions of an auxiliary eigenvalue problem is studied for physical applications in the manuscript \cite{physics}. }We study the approximation properties of this space as well as prove convergence of the Dirichlet spectral-Galerkin method for computing the zero-index transmission eigenvalues and provide error estimates. We will then provide some numerical examples in two dimensions to show that the proposed spectral method is effective for computing the eigenvalues. Once we have a method to approximate the eigenvalues we will turn our attention to estimating the refractive index for either large or small valued conductivity parameters.

\section{The Zero-Index Transmission Eigenvalues}\label{problem-statement}
This section focuses on the variational formulation of the zero-index transmission eigenvalue problem \eqref{zi-teprob3}. In particular, we study the associated solution operator. The analysis of the solution operator will be used in the convergence analysis of the spectral method. We define the variational space for \eqref{zi-teprob3} as $H^2(D) \cap H^1_0(D)$  where 
$$H^2(D) =\big\{ \varphi \in L^2(D) \, : \, \partial_{x_i} \varphi \quad  \text{and} \quad \partial_{x_i x_j} \varphi \in L^2(D) \, \text{ for } \, i,j =1, \cdots, d \big\} $$
and 
$$H^1_0(D) =\big\{ \varphi \in L^2(D) \, : \, \partial_{x_i} \varphi  \in L^2(D) \,\,  \text{ for } \,\,  i =1, \cdots , d \,\, \text{ with } \,\, \varphi|_{\partial D}=0 \big\}.$$
From \cite{two-eig-cbc} we have that the equivalent variational form for the  zero-index transmission eigenvalue problem \eqref{zi-teprob3} is given by the values $k \in \C$ such that there is a nontrivial $w \in H^2(D) \cap H^1_0(D)$ satisfying
\begin{align}
a(w,\varphi)=k^2b(w,\varphi)  \quad \text{ for all } \,\,\, \varphi \in H^2(D) \cap H^1_0(D). \label{zi-varform} 
\end{align} 
We will assume that the eigenfunctions are normalized with $\| w \|_{L^2(D)}=1$. 
The bounded sesquilinear forms on are defined by 
\begin{align}
a(w, \varphi)= \int\limits_{D} \frac{1}{n} \Delta w \, \Delta \overline{\varphi}\, \text{d}x \quad \text{and } \quad 
b( w, \varphi)=\int\limits_{D} \grad w \cdot \grad \overline{\varphi} \, \text{d}x -  \int\limits_{\partial D} \frac{1}{\eta} {\partial_\nu w} \, {\partial_\nu  \overline{\varphi}} \, \text{d}s. \label{forms}
\end{align} 
Recall, that we assume that there exists positive constants
$$ n_{\text{min}}  \leq n(x) \leq n_{\text{max}}  \quad \text{ a.e.} \, \, \,\,x \in \ov{D} \quad \text{ and } \quad  \eta_{\text{min}} \leq \eta(x)\leq \eta_{\text{max}}\quad \text{ a.e.} \, \, \,\, x \in \partial D.$$
We will study the variational formulation for the  zero-index transmission eigenvalue problem in this section. Even though this is a linear eigenvalue problem for $k^2$ notice that the sesquilinear form $b(\cdot \, , \cdot)$ is not sign definite due to the opposing signs in the definition. Which does not give a semi-norm on the variational space which is usually the case for standard elliptic eigenvalue problems.

The well-posedness estimate for the Poisson problem with zero trace along with the $H^2$ elliptic regularity estimate gives have that for any $H^2(D)$ function with zero trace $\|\Delta \cdot \|_{L^2(D)}$ is equivalent to the $\| \cdot \|_{H^2(D)}$. Therefore, we let 
$$X(D)=H^2(D) \cap H^1_0(D) \quad \text{ such that } \quad \| \cdot \|_{X(D)} = \|\Delta \cdot \|_{L^2(D)}.$$
Clearly, $X(D)$ is a Hilbert space with the associated inner-product. This implies that $a(\cdot \, , \cdot)$ is a coercive and Hermitian sesquilinear form on $X(D)$. This implies that $k=0$ is not a zero-index transmission eigenvalue. Now, by the Lax-Milgram Lemma we can define the solution operator $T: X(D) \to X(D)$ as 
\begin{align}
a\big(Tf,\varphi \big)=b(f,\varphi)  \quad \text{ for all } \,\,\, f, \varphi \in X(D). \label{def-T}
\end{align} 
From the definition of $T$ in \eqref{def-T} we have the following result. 

\begin{theorem}\label{T-compact}
Let the operator $T: X(D) \to X(D)$ be as defined by \eqref{def-T}. Then $T$ is an $a(\cdot \, , \cdot)$ self-adjoint compact operator and satisfies the estimate
$$\| Tf \|_{X(D)} \leq C \Big( \| f \|_{H^1(D)} + \|  {\partial_\nu f}  \| _{L^2(\partial D)} \Big).$$
\end{theorem}
\begin{proof}
Since $a(\cdot \, , \cdot)$ is a coercive and Hermitian sesquilinear form on $X(D)$ it is an equivalent inner-product on $X(D)$. Therefore, we have that for all $f, \varphi \in X(D)$
\begin{align*}
a\big(Tf,\varphi\big)=b(f,\varphi)=\ov{b(\varphi,f)}= \ov{a\big(T\varphi,f\big)}=a\big(f,T\varphi\big)
\end{align*} 
since $n$ and $\eta$ are real-valued. Proving that $T$ is $a(\cdot \, , \cdot)$ self-adjoint on $X(D)$. 
By the compact embedding of $H^{1/2}(\partial D)$ into $L^2(\partial D)$ and $H^{2}(D)$ into $H^1( D)$ the compactness of $T$ will follow immediately from the estimate. To prove the estimate notice that by \eqref{forms} and \eqref{def-T} we can conclude that 
\begin{align*}
\frac{1}{n_{\text{max}}} \| Tf \|^2_{X(D)} &\leq a\big(Tf,Tf\big)=b\big(f,Tf\big)  \\
							    &\leq  \Big( \| f \|_{H^1(D)} \| Tf \|_{H^1(D)} + \frac{1}{\eta_{\text{min}}} \|  {\partial_\nu f}  \| _{L^2(\partial D)} \|  {\partial_\nu Tf}  \| _{L^2(\partial D)} \Big)
\end{align*} 
where we have used the bounds on the coefficients. By appealing to the Trace Theorem and the continuous embedding of $H^{2}(D)$ into $H^1( D)$ we further have that 
$$\frac{1}{n_{\text{max}}} \| Tf \|^2_{X(D)} \leq C \Big( \| f \|_{H^1(D)} + \|  {\partial_\nu f}  \| _{L^2(\partial D)} \Big)\| Tf \|_{X(D)} $$
proving the claim. 
\end{proof}

Notice that since $T$ is a self-adjoint operator on the Hilbert space $X(D)$ with the $a(\cdot \, , \cdot)$ inner-product the Hilbert-Schmidt Theorem implies that there exists infinitely many eigenvalues counting multiplicity $\mu \in \R$ for the operator $T$ such that 
$$Tw = \mu w \quad \text{ which implies that } \quad \mu =k^{-2}.$$
Note that since $T$ is not sign definite there can be complex transmission eigenvalues $k$ that are purely imaginary. In \cite{two-eig-cbc} it has been shown that there are infinitely many zero-index transmission eigenvalues $k \in \R$.  Also, note that again by the Hilbert-Schmidt Theorem we have that there are infinity many eigenfunctions $w$ that form an $a(\cdot \, , \cdot)$ orthonormal basis of $X(D)$.
\section{The Dirichlet Spectral-Galerkin Approximation}\label{conv-analysis} 
With the results given in the previous section we can prove the convergence and error estimates of the Dirichlet spectral-Galerkin approximation method of the zero-index transmission eigenvalue problem.  We will use the approximation space that is the span of finitely many Dirichlet eigenfunctions for the Laplacian in the domain $D$. To prove the convergence and error estimates we must show the approximation properties of this space and use the variational formulation \eqref{zi-varform} to show the convergence of the eigenvalues and eigenfunctions. Even though we focus on the approximation space of Dirichlet eigenfunctions similar analysis as in Section \ref{conv-eig} will work for any conforming approximation space such as the Legendre-Galerkin approximation which is used for the fourth order formulation of the classical transmission eigenvalue problem in \cite{spectraltev2}. 

 \subsection{Approximation Space }\label{approx}
 Here we will define the approximation space of Dirichlet eigenfunctions and study the approximation properties of the space. To begin, we let $\phi_j \in H^1_0(D)$ be the $j$th Dirichlet eigenfunction for the Laplacian and the corresponding eigenvalue $\lambda_j \in \R_{+}$ for the domain $D$. The Dirichlet eigenpair satisfy 
\begin{align}
- \Delta \phi_j = \lambda_j \phi_j \,\, \text{ in } \,\, D \quad \text{ where }  \quad  \| \phi_j\|_{L^2(D)}=1. \label{dirichlet-eig}
\end{align}
By again appealing to elliptic regularity we have that $\phi_j \in X(D)$. From \cite{two-eig-cbc} we have the following result. 
\begin{lemma} \label{dense-eig}
Let $\phi_j$ satisfy \eqref{dirichlet-eig}  then the span$\{ \phi_j \}_{j=1}^\infty$ is dense in $X(D)$. 
\end{lemma}
The eigenvalues are assumed to be arranged in non-decreasing order such that $0 < \lambda_j \leq \lambda_{j+1}$ for all $j \in\N$.  It is well known that the eigenfunctions $\{ \phi_j \}_{j=1}^\infty$ form an orthonormal basis of $L^2(D)$ which implies that for all $f \in L^2(D)$ 
\begin{align}
f = \sum\limits_{j=1}^{\infty} (f,\phi_j)_{L^2(D)} \phi_j \label{fouier-series}
\end{align}
as a convergent series in the $L^2(D)$ norm. This series representation will be used to show the approximation rates for this set of basis functions. To do so, we will show the convergence of this series in the $X(D)$ norm.

\begin{theorem}\label{varspace-rep} 
For all $f \in X(D)$ we have that \eqref{fouier-series} is convergent in the $X(D)$ norm. Moreover, 
$$\Delta f  = \sum\limits_{j=1}^{\infty} -\lambda_j (f,\phi_j)_{L^2(D)} \phi_j$$
and is an $L^2(D)$ norm convergent series which gives $\| f\|^2_{X(D)} =\sum\limits_{j=1}^{\infty} \lambda_j^{2} \big|(f,\phi_j)_{L^2(D)}\big|^2$.
\end{theorem}
\begin{proof}
The result follows by Lemma \ref{dense-eig} which gives that  $\psi_j=\phi_j / \lambda_j$ is an orthonormal basis for $X(D)$ by equation \eqref{dirichlet-eig} such that 
$$f = \sum\limits_{j=1}^{\infty} (f,\psi_j)_{X(D)} \psi_j$$
as a $X(D)$ convergent series. Then appealing to Green's 2nd Theorem we have that   
$$ (f,\psi_j)_{X(D)} = -(\Delta f,\phi_j)_{L^2(D)}=  - ( f,\Delta \phi_j)_{L^2(D)} =  \lambda_j ( f, \phi_j)_{L^2(D)}.$$
Therefore, we obtain that 
$$f = \sum\limits_{j=1}^{\infty} (f,\phi_j)_{L^2(D)} \phi_j $$
is a convergent series in the $X(D)$ norm. Then applying the Laplacian term by term to the series representation proves the claim. 
\end{proof}

By the series representation \eqref{fouier-series} for any $f\in L^2(D)$ along with \eqref{dirichlet-eig} we can define the powers of the Laplacian $\Delta^m$ for $m \in \N$ by 
\begin{align} 
\Delta^m f  = \sum\limits_{j=1}^{\infty} (-\lambda_j)^m (f,\phi_j)_{L^2(D)} \phi_j. \label{power}
\end{align}
Note that this is often done to define (fractional) powers of an elliptic operator(see for {\color{black}example} \cite{frac-heat}). We will denote the domain of the $m$th power of the Laplacian in the set $L^2(D)$ as 
$$\mathscr{D} \big( \Delta^m \big) :=\big\{ f \in L^2(D) \, : \,  \Delta^m f \,\, \text{as defined in \eqref{power} is an $L^2(D)$ convergent series} \big\}. $$ 
Therefore, we have that $\mathscr{D} \big( \Delta^m \big)$ is a Hilbert space with the associated norm  
\begin{align}
\| f \|^2_{\mathscr{D} ( \Delta^m )} = \sum\limits_{j=1}^{\infty} \lambda_j^{2m} \big|(f,\phi_j)_{L^2(D)}\big|^2 <\infty. \label{laplacian-m}
\end{align} 
By Theorem \ref{varspace-rep} we have that $X(D) \subseteq \mathscr{D} ( \Delta )$. 

 For our spectral approximation method we will take the conforming computational subspace of $X(D)$ to be given by 
 $$X_N (D) = \text{span}\big\{ \phi_j\big\}_{j=1}^N \quad \text{ for some fixed } \quad N \in \N.$$
 A key ingredient to determining the approximation rate for this set of basis functions is Weyl's law for the Dirichlet eigenvalues(see \cite{weyl-law}). Weyl's law states that there exists two constants $c_1 , c_2 >0$ independent of $j$ such that 
 $$c_1 j^{2/d} \leq \lambda_j \leq c_2 j^{2/d}\quad \text{ for all } \quad j \in \N$$
where again the dimension $d=2,3$. We now consider the $L^2(D)$ projection onto the approximation space $X_N (D)$ which we denote $\Pi_N :X(D) \to X_N (D)$ and is given by
 $$\Pi_N f =  \sum\limits_{j=1}^{N}  (f,\phi_j)_{L^2(D)} \phi_j \quad \text{for some fixed} \quad N \in \N.$$ 
 It is clear that we have the point-wise convergence
 $$\big\| (I-\Pi_N)f \big\|^2_{X(D)} =  \sum\limits_{j=N+1}^{\infty} \lambda_j^{2} \big|(f,\phi_j)_{L^2(D)}\big|^2 \to 0 \quad \text{as} \quad N \to \infty$$ 
 for any $f \in X(D)$ by Theorem \ref{varspace-rep}. We now give a convergence estimate for any $f \in X(D) \cap \mathscr{D} \big( \Delta^m \big)$ where again we have that $\mathscr{D} \big( \Delta^m \big)$ is the subspace of $L^2(D)$ function such that equation \eqref{laplacian-m} is satisfied. 

\begin{theorem}\label{convrate}
Assume that $f \in X(D) \cap \mathscr{D} \big( \Delta^m \big)$ for some $m\geq 2$ then 
$$\big\| (I-\Pi_N) f \big\|_{X(D)} \leq \frac{C}{(N+1)^{2(m-1)/d}} \| f \|_{\mathscr{D} ( \Delta^m )} .$$
\end{theorem} 
 \begin{proof}
 To prove the claim we estimate 
 $$\big\| (I-\Pi_N) f \big\|^2_{X(D)} = \sum\limits_{j=N+1}^{\infty} \lambda_j^{2} \big|(f,\phi_j)_{L^2(D)}\big|^2 \leq \frac{1}{{\lambda_{N+1}^{2(m-1)}}}\sum\limits_{j=N+1}^{\infty} \lambda_j^{2m} \big|(f,\phi_j)_{L^2(D)}\big|^2$$
 where we have used the series representation in Theorem \ref{varspace-rep}. Now appealing to Weyl's law and by  \eqref{laplacian-m} we have that 
 $$\big\| (I-\Pi_N)f \big\|^2_{X(D)}  \leq \frac{C}{(N+1)^{4(m-1)/d}}\sum\limits_{j=1}^{\infty} \lambda_j^{2m} \big|(f,\phi_j)_{L^2(D)}\big|^2 
 = \frac{C}{(N+1)^{4(m-1)/d}}  \| f \|^2_{\mathscr{D} ( \Delta^m )}$$
 proving the estimate.
 \end{proof}

\subsection{Convergence and Error Estimates}\label{conv-eig}
 In this section, we will establish the Dirichlet spectral-Galerkin method for the zero-index transmission eigenvalue problem. In our analysis we will use the approximation space $X_N(D)$ defined as the span of the first $N$  Dirichlet eigenfunctions for the Laplacian in $D$. Similar results can be established  by using other conforming approximation subspaces such as a finite element approximation space of piecewise polynomials. Using the convergence analysis for the approximation space in the previous section we are now ready to prove the convergence on the spectral approximation.

 To begin, we will first consider the approximation for the operator $T$ defined in \eqref{def-T} by the $L^2(D)$ projection of $T$ onto the space $X_N (D)$. We will show that this approximation converges in the operator norm. This fact will be used to prove convergence and error estimates for the approximation of the eigenvalues. 
 \begin{theorem}
 Let the operator $T: X(D) \to X(D)$ be as defined by \eqref{def-T} and $\Pi_N :X(D) \to X_N (D)$ be the $L^2(D)$ projection onto $X_N (D)$. 
 Then $\Pi_N T \to T$ as $N \to \infty$ in the operator norm.
 \end{theorem}
 \begin{proof}
 Since the operator $T$ defined in equation \eqref{def-T} is compact by Theorem \ref{T-compact} we have that the point-wise convergence of $\Pi_N$ to the Identity operator on $X(D)$ implies that 
 $$\big\| (I-\Pi_N)T\big\|_{X(D) \mapsto X(D)}\to 0 \quad \text{as} \quad N \to \infty.$$
 Proving the claim. 
\end{proof}

We now define the spectral approximation of the zero-index transmission eigenvalue problem as find the values $k_N \in \C$ such that there is a nontrivial $w_N \in X_N(D)$ satisfying
\begin{align}
a(w_N,\varphi_N)=k_N^2b(w_N,\varphi_N)  \quad \text{ for all } \,\,\, \varphi_N \in X_N(D) \label{spec-varform} 
\end{align} 
where the sesquilinear forms $a( \cdot \, , \, \cdot)$ and $b( \cdot \, , \, \cdot)$ are defined by \eqref{forms}. Here we again assume that the eigenfunctions are normalized such that $\| w_N \|_{L^2(D)}=1$. Therefore, just as in the continuous case we can define the spectral approximation of the solution operator as 
operator $T_N: X(D) \to X_N(D)$ such that for any $f \in X(D)$  
\begin{align}
a\big(T_Nf,\varphi_N\big)=b(f,\varphi_N)  \quad \text{ for all } \,\,\, \varphi_N \in X_N(D). \label{def-TN}
\end{align} 
Since the dimension of the range of $T_N$ is finite we have that it is compact. 
It is also clear that $T_N$ restricted to $X_N(D)$ is an $a( \cdot \, , \, \cdot)$ self-adjoint operator on $X_N(D)$. This gives that $T_N$ has $N$ eigenvalues counting multiplicity. Arguing similarly as in Section \ref{problem-statement} we have that the eigenpair $(k_N , w_N) \in \C \times X_N(D)$ satisfying \eqref{spec-varform} is the eigenpair for the spectral approximation of the solution operator such that $$T_N w_N = k^{-2}_N w_N.$$

In order to attain the convergence as well as an error estimate we will study the convergence of the spectral approximation of the solution operator as $N \to \infty$. Therefore, by appealing to Galerkin orthogonality 
$$a\big(T f - T_Nf,\varphi_N\big) = 0 \quad \text{ for all } \,\,\, \varphi_N \in X_N(D)$$
and Cea's Lemma(\cite{numerics-book} page 372) we have that 
$$\big\| Tf - T_N f \big\|_{X(D)} \leq C \big\| Tf - v_N \big\|_{X(D)} \quad \text{ for any } \,\,\, v_N  \in X_N(D)$$
and for all $f \in X(D)$. From the above estimate we conclude that
$$\big\| Tf - T_N f \big\|_{X(D)} \leq C  \big\| (I-\Pi_N)T f \big\|_{X(D)}$$
where again $\Pi_N$ is the $L^2(D)$ projection onto the approximation space $X_N(D)$.  This analysis gives the following result. 

\begin{theorem} \label{discrete-eig-conv}
Let $(k , w)$ and $(k_N , w_N)$ be the $j$th egienpair for \eqref{zi-varform} and \eqref{spec-varform} respectively. Then as $N \to \infty$ we have that $k_N \to k$ and $w_N \to w$ in $X(D)$.  
\end{theorem}
\begin{proof}
This result follows from the fact that 
$$\big\|T -T_N \big\|_{X(D) \mapsto X(D)} \leq  C \big\| (I-\Pi_N)T  \big\|_{X(D)  \mapsto X(D)} \to 0 \quad \text{ as } \quad N \to \infty$$
and then appealing to the results in \cite{osborn}. 
\end{proof}
 
Now that we have established the convergence of the spectral approximation our next step it to determine the convergence rate. To this end, we will argue similarly to Theorem \ref{convrate} along with the using variational formulations \eqref{zi-varform} and \eqref{spec-varform}. Simple calculations give that for $(k , w)$ and $(k_N , w_N)$ being the $j$th egienpair for \eqref{zi-varform} and \eqref{spec-varform} then 
\begin{align}
a\big(w_N-w , w_N -w\big) -k^2 b\big(w_N-w , w_N -w\big) = \big( k^2_N -k^2 \big)b\big(w_N , w_N\big)\label{eig-difference}
\end{align}
for any $N$. The equality \eqref{eig-difference} will be used to establish the convergence rate for the eigenvalues. 
So we need to establish that the sequence $\big| b\big(w_N , w_N\big) \big| $ is bounded below.
 Therefore, we present the following result. 
\begin{theorem} \label{inf-b}
Let $(k_N , w_N)$ be the $j$th egienpair for \eqref{spec-varform}. Then there is a constant $\beta>0$ independent of $N$ such that 
$$\inf\limits_{N \in \N} \Big\{ \big| b\big(w_N , w_N\big) \big| \Big\} \geq \beta.$$  
\end{theorem}
\begin{proof}
To prove the claim we proceed by way of contradiction. To this end, assume no such $\beta$ exists, then we can extract a subsequence still denoted with $N$ such that 
$$\big| b\big(w_N , w_N\big) \big| \to 0 \quad \text{ as } \quad N \to \infty.$$ 
By the continuity of the sesquilinear form $b( \cdot \, , \, \cdot)$ and Theorem \ref{discrete-eig-conv} we obtain 
$$\big| b\big(w_N , w_N\big) \big| \to \big| b\big(w , w\big) \big| \quad \text{ as } \quad N \to \infty$$ 
where $w$ is an eigenfunction corresponding to \eqref{zi-varform}. The variational formulation implies that $a(w,w)=0$ which contradicts the fact that $\| w \|_{L^2(D)}=1$ since $a( \cdot \, , \, \cdot)$ defines an inner-product on $X(D)$. Proving the claim. 
\end{proof}

From Theorem \ref{inf-b} we can conclude that for $(k , w)$ and $(k_N , w_N)$ being the $j$th egienpair for \eqref{zi-varform} and \eqref{spec-varform} respectively then there is a $C>0$ independent of $N$ where
$$\big| k_N^2-k^2 \big| \leq C \big\| w_N -w \big\|^2_{X(D)}.$$
Note that we have used \eqref{eig-difference} and the boundedness of sesquilinear forms $a( \cdot \, , \, \cdot)$ and $b( \cdot \, , \, \cdot)$. In order to obtain the error estimate for the spectral approximation of the zero-index transmission eigenvalues we must estimate the error in the Galerkin approximation in the approximation space $X_N(D)$ on the eigenspace corresponding to $k$. To this end, we will denote the eigenspace corresponding to the zero-index transmission eigenvalue $k$ by $E(k)$. It is clear that $E(k) \subset X(D)$ is finite dimensional and any $u \in E(k)$ satisfies $Tu = k^{-2} u$. With this we can now prove the error estimate. 

\begin{theorem} \label{eig-convrate2}
Let $k $ and $k_N$ be the $j$th eigenvalues for \eqref{zi-varform} and \eqref{spec-varform} respectively. If the corresponding eigenspace $E(k) \subset \mathscr{D} \big( \Delta^m \big)$ for $m\in \N$ then there is a $C>0$ independent of $N$ such that  
$$ \big| k_N^2-k^2 \big| \leq \frac{C}{ (N+1)^{4(m-1)/d}}\sup\limits_{u \in E(k) \, ,\,  \| u \|_{X(D)} =1} \big\| (I -\Pi_N)u \big\|^2_{\mathscr{D} ( \Delta^m )}  $$  
where $\Pi_N :X(D) \to X_N (D)$ is the $L^2(D)$ projection onto $X_N (D)$.
\end{theorem}
\begin{proof}
Notice that according to the analysis in \cite{babuska-osborn}, we have that 
$$\| w_N -w \|^2_{X(D)} \leq C \big\|(T -T_N) \big\|^2_{E(k) \mapsto X(D)}.$$
Therefore, by the definition of the norm $\big\| \cdot \big\|_{E(k) \mapsto X(D)}$ we can estimate 
\begin{align*}
\big\|(T -T_N) \big\|^2_{E(k) \mapsto X(D)} &=\sup\limits_{u \in E(k) \, ,\,  \| u \|_{X(D)} =1} \big\|(T -T_N)u \big\|^2_{X(D)}\\
			&\leq C \sup\limits_{u \in E(k) \, ,\,  \| u \|_{X(D)} =1} \big\|(I -\Pi_N)Tu \big\|^2_{X(D)}.
\end{align*}
We further have that 
\begin{align*}
\big\|(T -T_N) \big\|^2_{E(k) \mapsto X(D)} &\leq C \sup\limits_{u \in E(k) \, ,\,  \| u \|_{X(D)} =1} \big\|(I -\Pi_N)Tu \big\|^2_{X(D)}\\
			&=C \sup\limits_{u \in E(k) \, ,\,  \| u \|_{X(D)} =1} \sum\limits_{j=N+1}^{\infty} \lambda_j^{2} \big|(Tu,\phi_j)_{L^2(D)}\big|^2\\
			&=C|k|^{-4}\sup\limits_{u \in E(k) \, ,\,  \| u \|_{X(D)} =1} \sum\limits_{j=N+1}^{\infty} \lambda_j^{2} \big|(u,\phi_j)_{L^2(D)}\big|^2.
\end{align*}
Where we have used the fact that $u \in E(k)$. Now, since we have assumed that $E(k) \subset \mathscr{D} \big( \Delta^m \big)$ we further estimate just as in Theorem \ref{convrate}  			
\begin{align*}
\big\|(T -T_N) \big\|^2_{E(k) \mapsto X(D)} &\leq C\lambda_{N+1}^{-2(m-1)} \sup\limits_{u \in E(k) \, ,\,  \| u \|_{X(D)} =1} \sum\limits_{j=N+1}^{\infty} \lambda_j^{2m} \big|(u,\phi_j)_{L^2(D)}\big|^2\\
			&\leq  C (N+1)^{-4(m-1)/d} \sup\limits_{u \in E(k) \, ,\,  \| u \|_{X(D)} =1} \big\| (I -\Pi_N)u \big\|^2_{\mathscr{D} ( \Delta^m )}
\end{align*}
which we obtain by appealing to Weyl's law. This estimate give the convergence rate proving the claim. 
\end{proof} 
 
We end this section by noting that the convergence rate for the eigenfunctions are the square root of the convergence rate for the eigenvalues. This is clear from the proof of Theorem \ref{eig-convrate2} since estimates for $\| w_N -w \|^2_{X(D)}$ are used to derive the estimates for the eigenvalue convergence rate.

 \begin{theorem} \label{eigfunc-convrate}
Let $w$ and $w_N$ be the $j$th eigenfunctions for \eqref{zi-varform} and \eqref{spec-varform} respectively. If the corresponding eigenspace $E(k) \subset \mathscr{D} \big( \Delta^m \big)$ for  $m \in \N$ then there is a $C>0$ independent of $N$ such that  
$$ \| w_N-w \|_{X(D)} \leq \frac{C }{ (N+1)^{2(m-1)/d}}  \sup\limits_{u \in E(k) \, ,\,  \| u \|_{X(D)} =1}\big\| (I -\Pi_N)u \big\|_{\mathscr{D} ( \Delta^m )}  $$
where $\Pi_N :X(D) \to X_N (D)$ is the $L^2(D)$ projection onto $X_N (D).$  
\end{theorem}
\begin{proof}
The result follows from the analysis given in the proof of Theorem \ref{eig-convrate2}.
\end{proof}

\section{Numerical Examples}\label{numerics}  
In this section, we provided some numerical examples of computing the zero-index transmission eigenvalues via the  Dirichlet spectral-Galerkin method studied in the previous sections. For simplicity, will assume that the domain $D$ is the unit ball in $\R^2$. We refer to  \cite{spectraltev3} for the approximation of the classical transmission eigenvalues via a Spectral Element Method for a spherically stratified media. We will check the accuracy of our Dirichlet spectral-Galerkin method by comparing to the eigenvalues computed by separation of variables for constant refractive index $n$ and conductivity $\eta$. The computations in this section are motivated by the work in \cite{gpinvtev} where the authors studied the convergence of the spectral-Galerkin method for computing the classical transmission eigenvalues where the basis functions in the approximation space are the eigenfunctions for the bilaplacian with {\color{black} zero clamped} plate boundary conditions. We will also consider estimating the refractive index $n$ for $\eta$ either small or large but unknown. To do so, we will use the convergence results given in \cite{two-eig-cbc} where it is shown that as the conductivity tends to zero or infinity one obtains an eigenvalue problem that only depends on the refractive index. The estimation of the refractive index from the scattering data has been studied in \cite{isp-n1,isp-n2} to name a few examples. In these papers the authors show that the far and near field data can be used to recover the classical transmission eigenvalues and use monotonicity results to recover a constant refractive index or estimate a variable refractive index. Here we numerically study this problem for the zero-index transmission eigenvalues.

 \subsection{Computing the Eigenvalues}\label{eig-approx}
We are now ready to compute the zero-index transmission eigenvalues using the Dirichlet spectral-Galerkin approximation presented in Section \ref{conv-analysis}. All of our experiments are done with the software \texttt{MATLAB} 2018a on an iMac with a 4.2 GHz Intel Core i7 processor with 8GB of memory. To begin, we will describe an effective method for computing the approximated eigenvalues that satisfy \eqref{spec-varform}. Since the domain $D$ is given by the unit circle we can determine the basis functions from separation of variables. Therefore, the basis functions are taken to be
$$\phi_{j}(r,\vartheta) = J_{p} \left(\sqrt{\lambda_{p,q}} \, r\right) \cos(p \vartheta)\quad \text{with index} \quad j=j(p,q) \in \N.$$
Here square root of the eigenvalues $\sqrt{\lambda_{p,q}}$ corresponds to the $q$th positive root of the $p$th first kind Bessel function denoted $J_p$ for all $p\in \N \cup\{0\}$ and $q \in \N$. 

In the following numerical examples we take 24 basis functions where $0\leq p\leq 5$ and $1\leq q \leq 4$ which will give that the approximation space is defined as 
$$ \text{Span} \Big\{ \phi_{j(p,q)}(r,\vartheta)  \Big\}_{p=0\, ,\, q=1}^{p=5 \, ,\, q=4} \subset X(D).$$
For our Dirichlet spectral-Galerkin approximation we will solve \eqref{spec-varform} for $w_N$ in the aforementioned approximation space. This gives that the  approximation of the eigenfunctions will  have the form
$$w_N (x) = \sum\limits_{j=1}^{24} \omega_j \phi_j (x).$$
Therefore, the spectral approximation of the eigenvalues $k_N$ satisfying \eqref{spec-varform} correspond to the eigenvalues for the matrix equation  
\begin{align}
\left( {\bf A}-k_N^2{\bf B}\right) \vec{\omega} =0 \quad \text{ where } \quad \vec{\omega}=\big(\omega_1, \cdots ,\omega_{24} \big)^{\top} \neq 0. \label{g-eig}
\end{align}
We have that the $24 \times 24$ Galerkin mass and stiffness matrices in the Spectral approximation \eqref{spec-varform} are given by
$${\bf A}_{i,j}  = a(\phi_i , \phi_j) \quad \text{ and } \quad {\bf B}_{i,j} =  b(\phi_i , \phi_j) \quad \text{ for } \,\, i,j =1, \cdots , 24.$$
The integrals can be simplified by using \eqref{dirichlet-eig} to evaluate the sesquilinear forms $a( \cdot \, , \, \cdot)$ and $b( \cdot \, , \, \cdot)$ giving that 
$$a(\phi_i , \phi_j) =  \lambda_i \lambda_j \int\limits_{D} \frac{1}{n(x)}  \phi_i (x) \, {\phi}_j(x) \, \text{d}x $$
and
$$b(\phi_i , \phi_j) =\lambda_i  \int\limits_{D}  \phi_i (x) \, {\phi}_j(x) \, \text{d}x - \int\limits_{\partial D} \frac{1}{\eta(x) }\,  {\partial_\nu \phi_i (x)}  \, {\partial_\nu  {\phi}_j (x) }\, \text{d}s.$$
Notice that we have used that for the Dirichlet eigenfunctions $\phi_i$ and $\phi_j$ we have the following integral identities  
$$\int\limits_{D} \frac{1}{n(x)} \Delta  \phi_i (x) \, \Delta{\phi}_j(x) \, \text{d}x =  \lambda_i \lambda_j \int\limits_{D} \frac{1}{n(x)}  \phi_i (x) \, {\phi}_j(x) \, \text{d}x$$
and 
$$ \int\limits_{D} \grad \phi_i (x) \cdot \grad  {\phi}_j(x) \, \text{d}x =  \lambda_i \int\limits_{D} \phi_i (x) \,  {\phi}_j(x) \, \text{d}x \quad \text{for any} \quad i,j \in \N. $$ 
In our calculations, we use the fact that the Dirichlet eigenfunctions are orthogonal in $L^2(D)$ which gives that the volume integral in $b(\phi_i , \phi_j)$ corresponds to a diagonal matrix. 

To compute the Galerkin matrices we implement a 2d Gaussian quadrature method. In the numerical examples the integrals are written in polar coordinates where 12 quadrature points are used to evaluate the radial and angular parts of the integrals. The discretized eigenvalue problem \eqref{g-eig} is then solved using the `\texttt{polyeig}' command in \texttt{MATLAB} since \eqref{g-eig} is a quadratic eigenvalue problem for the parameter $k_N$. From \cite{two-eig-cbc} we have that for $n$ and $\eta$ constant the eigenvalues $k$ satisfy the transcendental equation
$$d_m(k):=k\sqrt{n}J_{|m|}' \big(k\sqrt{n}\big) -\big(\eta +|m|\big)J_{|m|}\big(k\sqrt{n}\big)=0 \quad \text{ for all } \quad m \in \Z.$$
Using the `\texttt{fzero}' command in \texttt{MATLAB} we can determine the exact zero-index transmission eigenvalues. 
In Figures \ref{dmplot1} and \ref{dmplot2} we plot the function $d_m(k)$ for the values of $m=0,1,2,3,4,5$ with $k \in [0,5]$. 
To validate our approximation method we compare the Approximation v.s. the Exact eigenvalues presented in Tables \ref{eig-compare1} and \ref{eig-compare2}. 
\begin{figure}[ht!]
\centering  
\includegraphics[width=11cm]{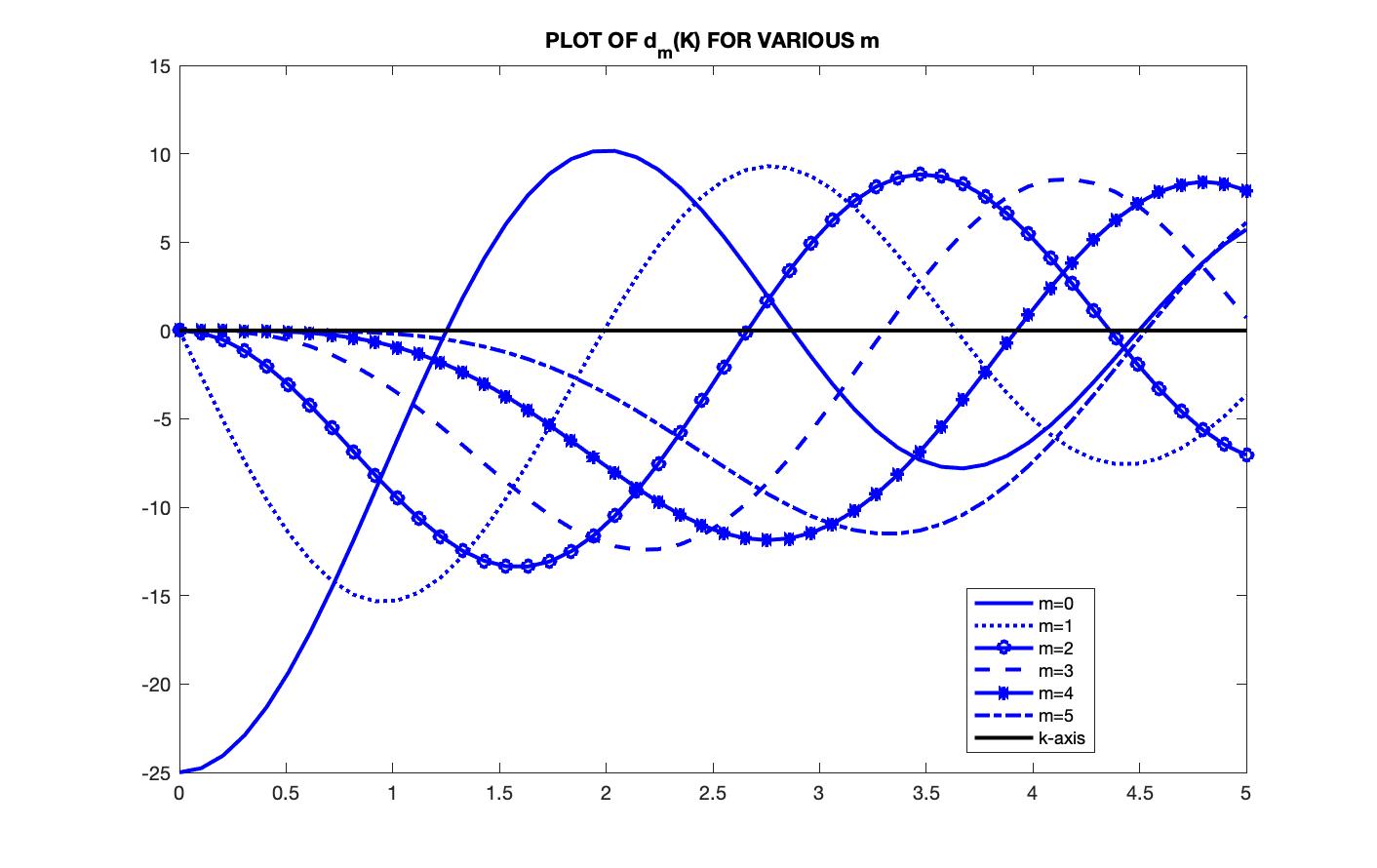}
\caption{Plot of the function  $d_m(k)$ for the values of $m=0,1,2,3,4,5$ with coefficient parameter $n=4$ and $\eta=25$.}\label{dmplot1}
\end{figure}

\begin{table}[ht!]
\centering  
\begin{tabular}{| c | c | c | c |} 
\hline                  
      &  Approximation     &  Exact  & Relative Error    \\ [0.5ex] 
\hline                  
$\, k_1\, $ & $1.25185566197 $ & $1.25132121108 $ &  $4.27 \times10^{-4}$ \\
$k_2$ & $1.99243796762 $ & $1.99043273844 $ &  $0.0010$ \\
$k_3$ & $2.878602256114 $ & $2.66364226350 $ & $0.0807$ \\
\hline 
\end{tabular}
\caption{Comparison of the Dirichlet spectral-Galerkin approximation v.s. Exact zero-index transmission eigenvalues  for $n=4$ and $\eta = 25$. }\label{eig-compare1}
\end{table}

\begin{figure}[ht!]
\centering  
\includegraphics[width=11cm]{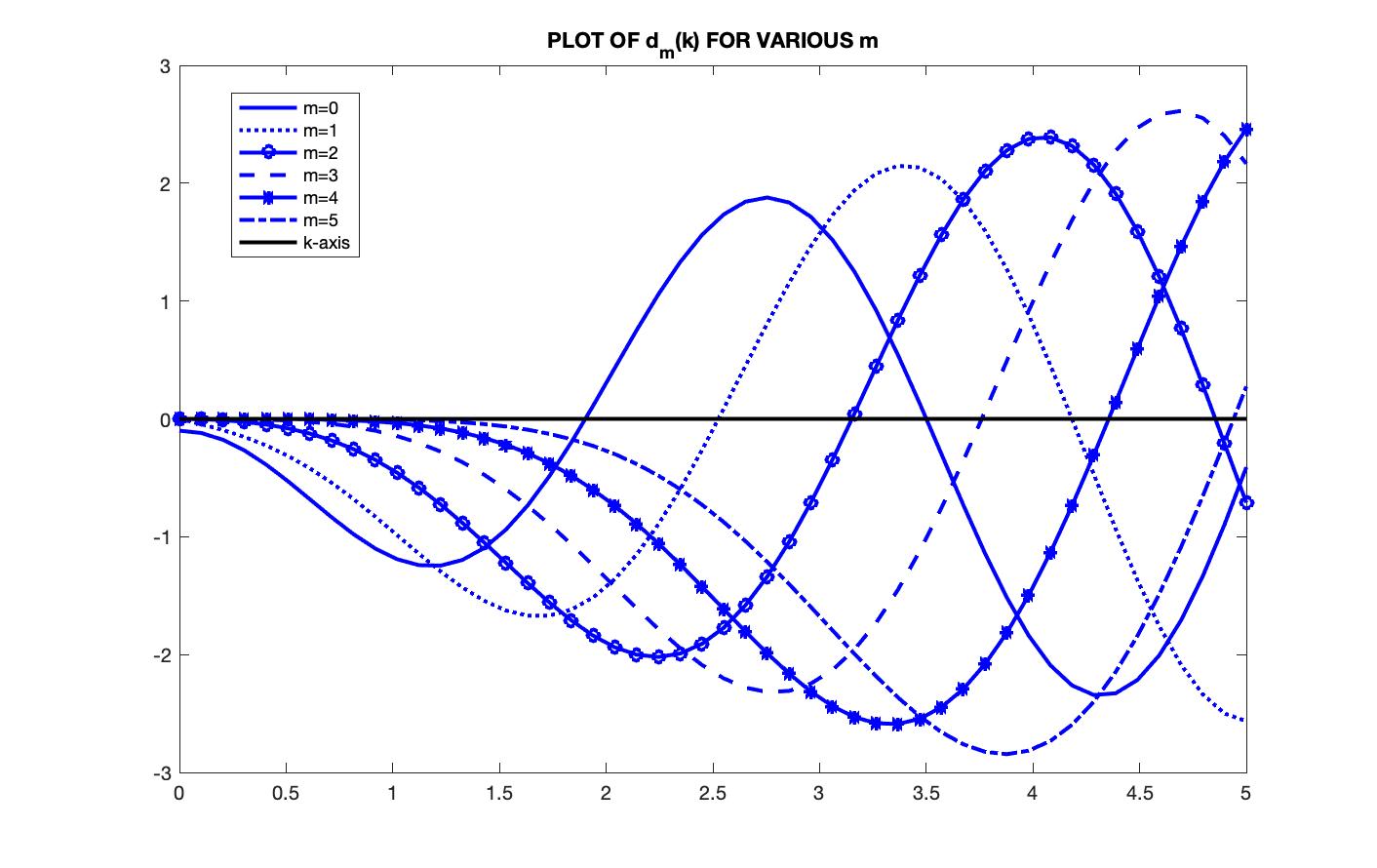}
\caption{Plot of the function  $d_m(k)$ for the values of $m=0,1,2,3,4,5$  with coefficient parameter $n=4$ and $\eta=1/10$.}\label{dmplot2}
\end{figure}

\begin{table}[ht!]
\centering  
\begin{tabular}{| c | c | c | c |} 
\hline                  
     & Approximation    &  Exact   & Relative Error    \\ [0.5ex] 
\hline                  
$\, k_1\, $ & $2.00233111434 $ & $1.90276223549 $ &  $0.0523$ \\
$k_2$ & $2.68333505931 $ & $ 2.55809498688 $ & $0.0490$ \\
$k_3 $ & $3.69381250080 $ & $3.18227361485 $ & $0.1607$ \\
\hline 
\end{tabular}
\caption{Comparison of the Dirichlet spectral-Galerkin approximation v.s. Exact zero-index transmission eigenvalues  for $n=4$ and $\eta = 1/10$. }\label{eig-compare2}
\end{table}

The plot of the relative error for the first eigenvalue $k_1$ is given where we let $n$ vary in the interval $[3,5]$ for $\eta = 25$ or $\eta = 1/10$. We use $d_0(k)$ to compute the exact first zero-index transmission eigenvalues. In Figure \ref{errorplot} we see that the relative error for $\eta = 25$  is on the order of $10^{-4}$ where as the relative error for $\eta = 1/10$  is on the order of $10^{-2}$. This seems to imply that the Dirichlet spectral-Galerkin approximation method is better suited for problems with larger conductivity. 
\begin{figure}[ht!]
\hspace{-0.2in}\includegraphics[width=8.5cm]{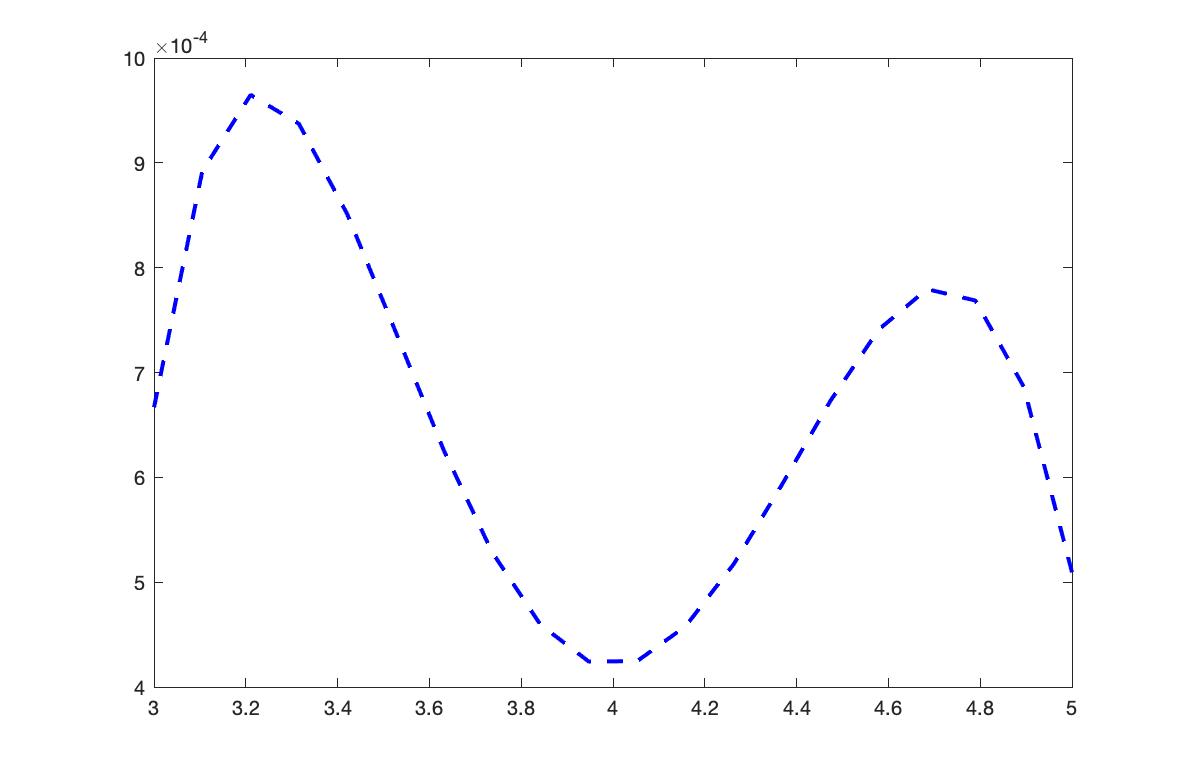}\includegraphics[width=8cm]{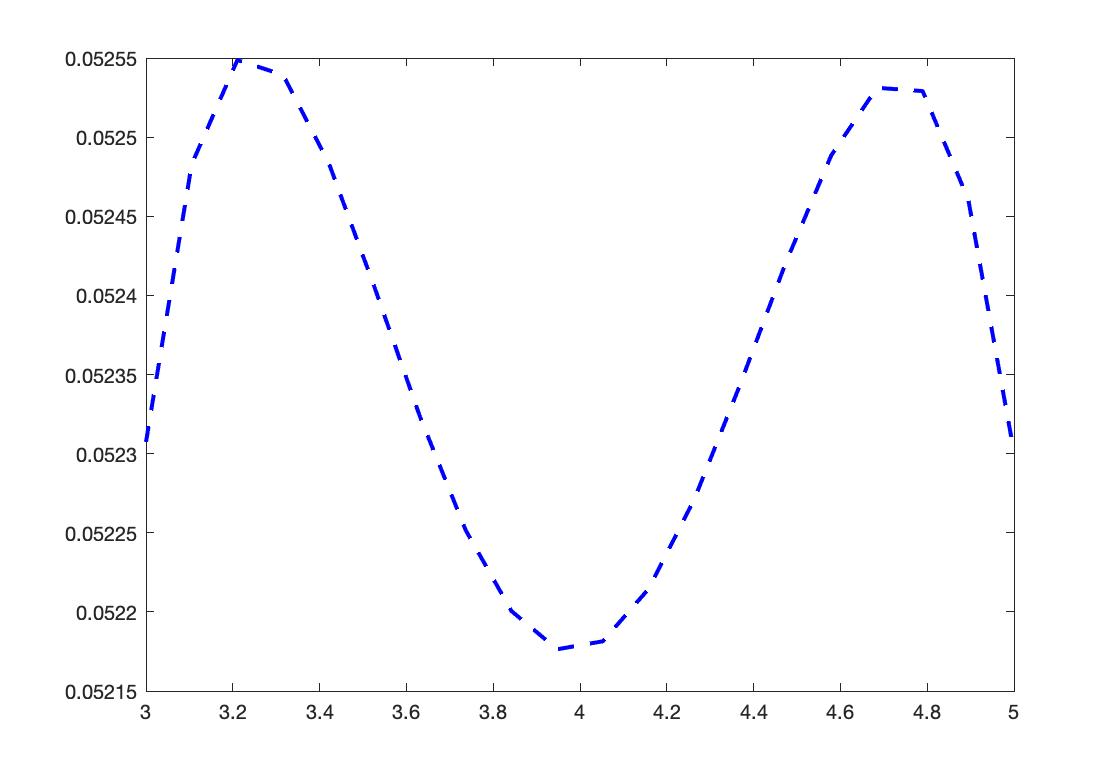}
\caption{Here is a plot of the relative error for the 1st eigenvalue for $n \in [3,5]$ for $\eta =25$ and $\eta=1/10$ on the left and right respectively. }\label{errorplot}
\end{figure}

We now consider computing the eigenvalues for variable coefficients. Here we take a smooth and a piece-wise constant refractive index defined as 
$$n_1=4+\text{exp}(-r^2) \quad \text{ and } \quad n_2= 4*\mathbbm{1}_{(0.25 \leq r < 1)} + 2*\mathbbm{1}_{(r<0.25)} $$
for a spherically stratified media. Here $\mathbbm{1}_{I}$ denotes the indicator function on the interval $I$. In Table \ref{eig-variable1} and \ref{eig-variable2} we report the first three real zero-index transmission eigenvalues computed via our approximation method for various conductivities. Here we take three different parameters $\eta$. Two of the conductivities are constants taken to be $25$ and $1/10$ while the third is a variable conductivity parameter 
$$\eta=1/\big(10+\sin^2(2\theta)\big).$$ 
Recall, that the analysis in Section \ref{conv-eig} also gives the convergence of the eigenfunctions in Theorems \ref{discrete-eig-conv} and \ref{eigfunc-convrate}. So, presented with Tables \ref{eig-variable1} and \ref{eig-variable2} are the plots for the first two corresponding eigenfunctions for the spherically stratified refractive indices $n_1$ and $n_2$ defined above. 

\begin{table}[ht!]
\centering  
\begin{tabular}{| c | c | c | c |} 
\hline                  
          & $\eta = 25$     &  $\eta = 1/10$   & $\eta=1/\big(10+\sin^2(2\theta)\big)$    \\ [0.5ex] 
\hline                  
$\, k_1\, $ & $1.13937194615 $ & $1.83076451238 $ &  $1.83137577764$  \\
$\, k_2\, $ & $1.82895096548 $ & $2.45629411834 $ &  $2.45677744275$ \\
$\, k_3\, $ & $2.63523894008 $ & $3.37709503556 $ &  $3.37739837186$ \\
\hline 
\end{tabular}
\includegraphics[width=7cm]{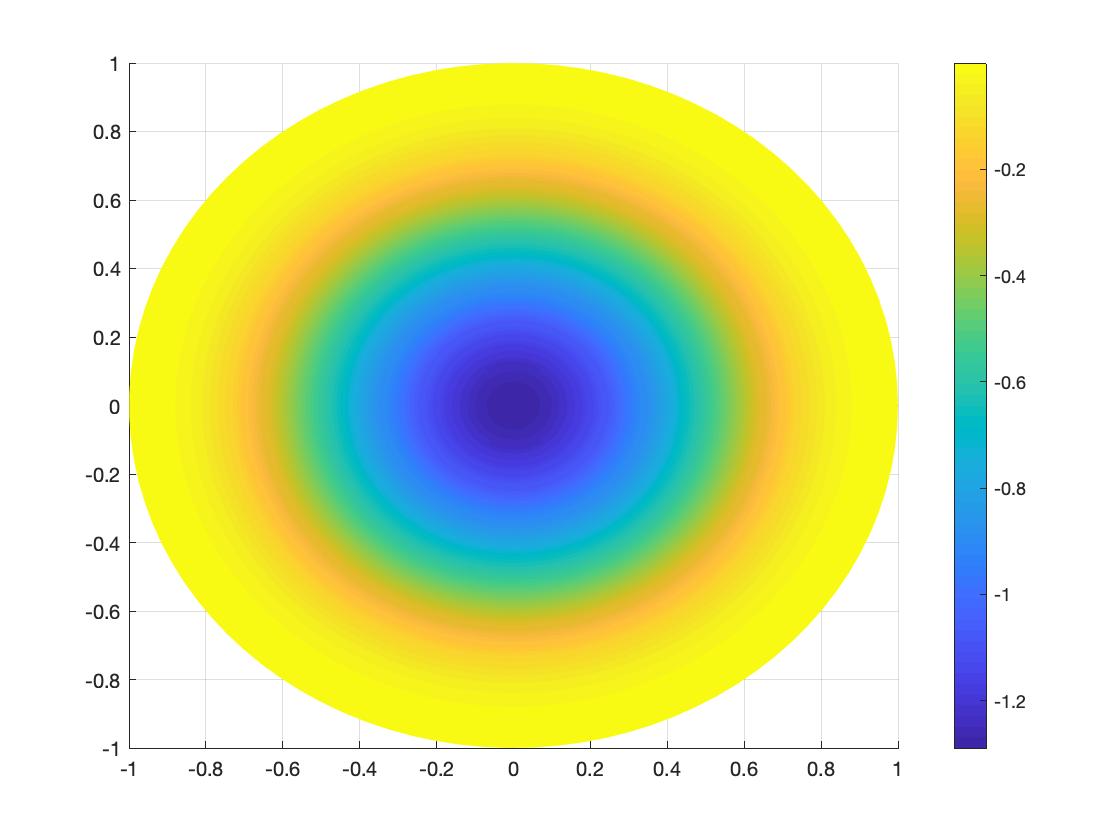}\includegraphics[width=7cm]{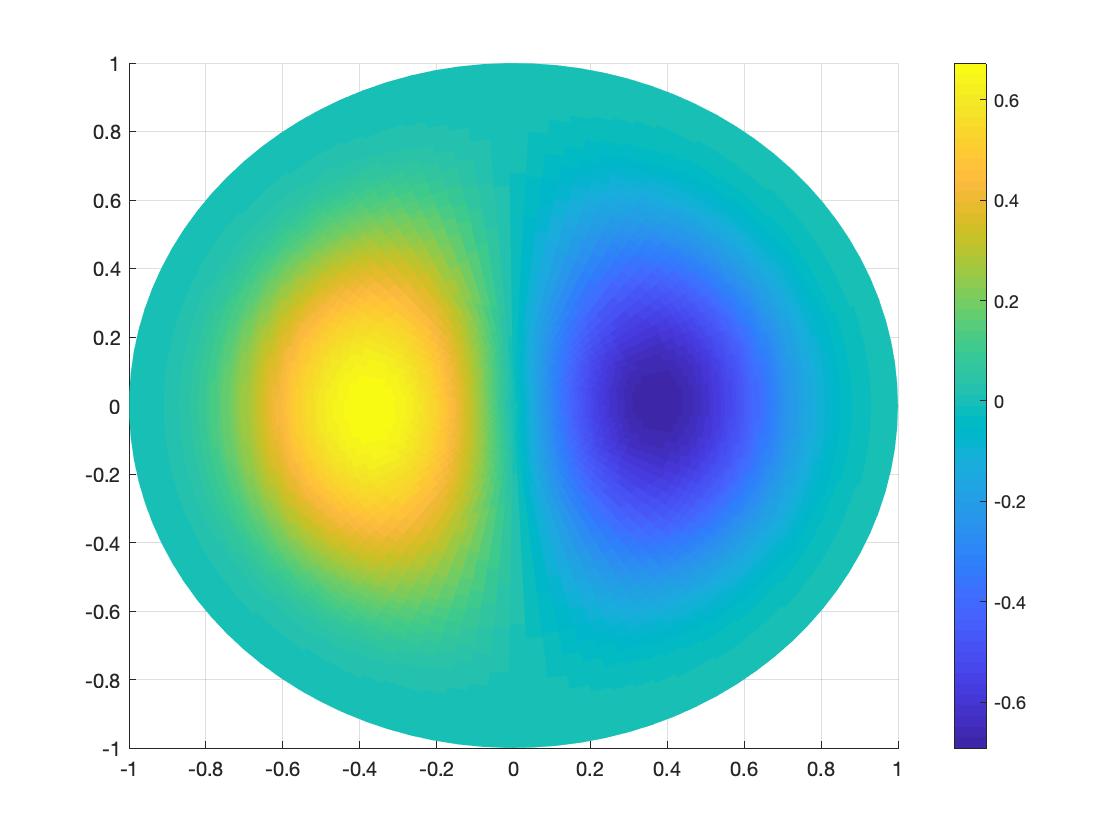}
\caption{The first three real zero-index transmission eigenvalues for $n_1$. We also plot the eigenfunctions corresponding to the eigenvalues $k_1$ and $k_2$ for $\eta=1/10$.}\label{eig-variable1}
\end{table}

\begin{table}[ht!]
\centering  
\begin{tabular}{| c | c | c | c |} 
\hline                  
       & $\eta = 25$    &  $\eta = 1/10$   & $\eta=1/\big(10+\sin^2(2\theta)\big)$    \\ [0.5ex] 
\hline                  
$\, k_1\, $ & $1.33698344835 $ & $2.20985718061 $ &  $2.21052625727$  \\
$k_2$ & $2.01851716957 $ & $2.77517970657 $ &  $2.77572095891$ \\
$k_3$ & $3.21273738555 $ & $4.14541092513 $ &  $4.14571008141$ \\
\hline 
\end{tabular}
\includegraphics[width=7cm]{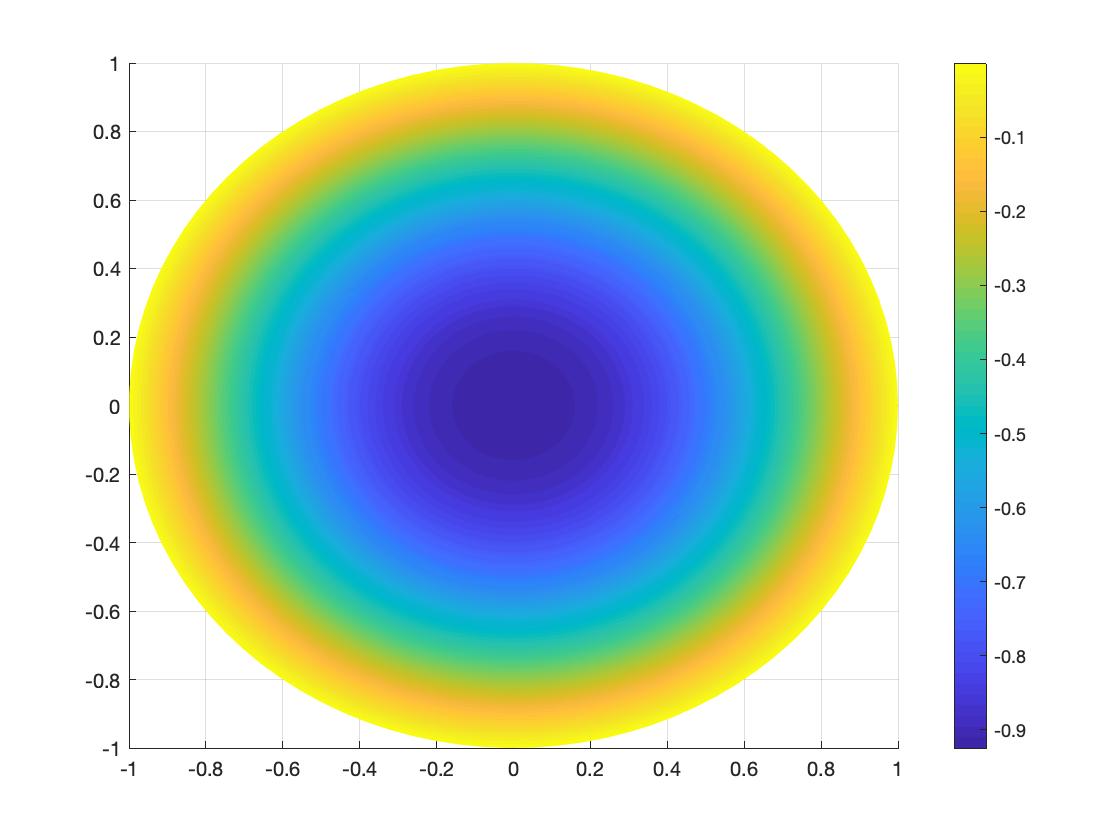}\includegraphics[width=7cm]{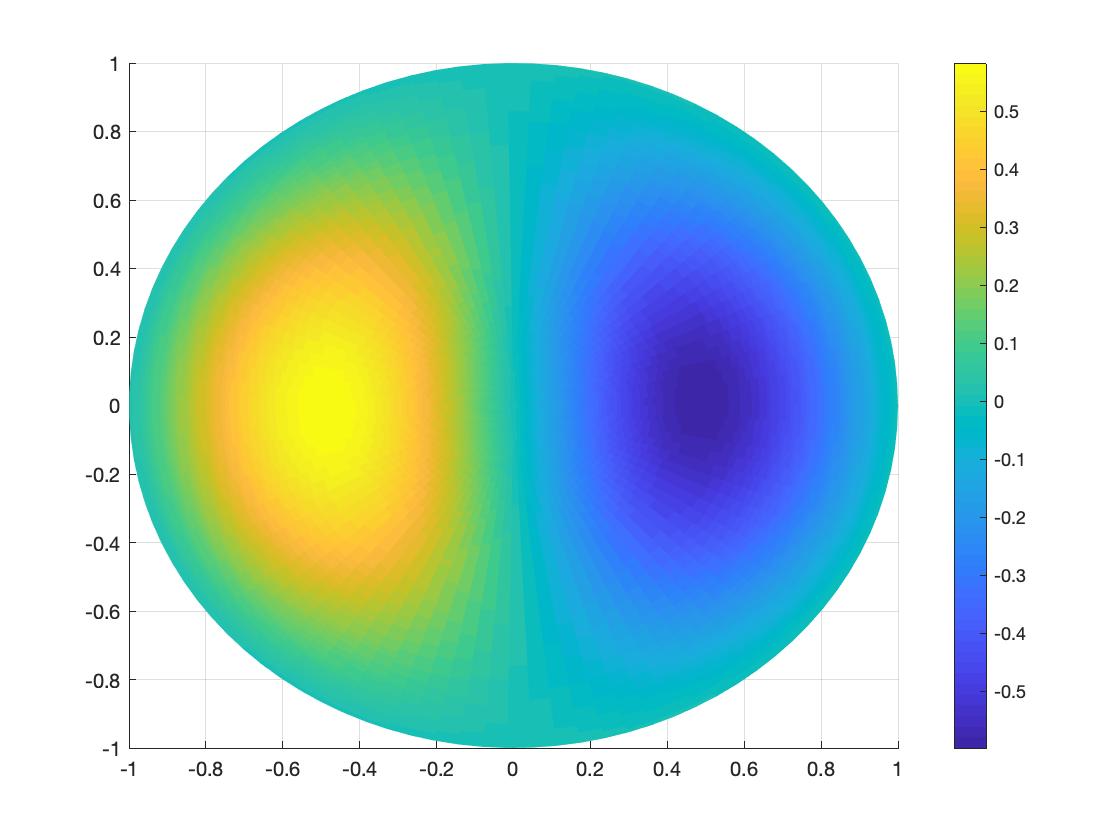}
\caption{The first three real zero-index transmission eigenvalues for $n_2$. We also plot the eigenfunctions corresponding to the eigenvalues $k_1$ and $k_2$ for $\eta=25$.}\label{eig-variable2}
\end{table}

Notice that the computed zero-index transmission eigenvalues are monotonically decreasing with respect to the coefficients which is predicted by the theory in \cite{two-eig-cbc}. We can see that for the first three real zero-index transmission eigenvalues that for $\eta = 25$ and $1/10$ we have 
$$k_j(n_1 , \eta)  \leq k_j(4,\eta) \leq k_j(n_2 , \eta).$$
Similarly comparing the reported eigenvalues we see the monotonicity with respect to the conductivity $\eta$ for various refractive indices.

 \subsection{Estimating the Refractive Index}\label{n-approx}
In this section, we present numerical examples for estimating the refractive index from the knowledge of the zero-index transmission eigenvalues. It has been shown in \cite{te-cbc,te-cbc2} that the classical transmission eigenvalues with a conductive boundary condition can be recovered from the scattering data via the Linear Sampling Method and the Inside-Out Duality, see \cite{cchlsm,armin} for details of these methods to recover the transmission eigenvalues. Therefore, we will assume that the zero-index transmission eigenvalues can be recovered from the scattering data and we wish to estimate $n$.

In order to estimate the refractive index from the zero-index transmission eigenvalues $k(n,\eta)$ we will restrict ourselves to the case where $\eta$ is either sufficiently large or small. The case for $\eta$ known was considered numerically in \cite{two-eig-cbc}. The limiting behavior of the zero-index transmission eigenvalues was studied in \cite{two-eig-cbc} as $\eta$ tends to zero or infinity. It has been shown that $k(n,\eta) \to \tau(n)$ as $\eta \to \infty$ where $\tau$ is a `modified' Dirichlet eigenvalue satisfying that there exist a nontrivial $v$ such that 
$$\Delta v +\tau^2 n v= 0 \,\, \text{ in } \,\, D \quad \text{ where } \quad v \in H^1_0(D)$$
or as $\eta \to 0$ where $\tau$ is a `modified' plate buckling eigenvalue satisfying that there exist a nontrivial $v$ such that 
$$\Delta \frac{1}{n} \Delta v +\tau^2 \Delta v=0  \,\, \text{ in } \,\, D \quad \text{ where } \quad v \in H^2_0(D).$$ 
This limiting behavior will allow use to estimate the refractive index without the knowledge of $\eta$ on the boundary. 

This gives that if it is known a prior that $\eta \ll 1$ or $\eta \gg 1$ then we can estimate the refractive index by finding the constant $n_{\text{approx}}$ such that $k_1(n,\eta) =  \tau_1(n_{\text{approx}})$ where $\tau_1$ is the first `modified' Dirichlet eigenvalue for $\eta \gg 1$ or the first `Modified' plate buckling eigenvalue for $\eta\ll 1$. It is known that $\tau_1$ depends monotonically on $n$ by the max-min principle \cite{eig-FEM-book}. Since $D$ is assumed to be known we can compute $\tau_1$ for any constant refractive index $n$ via the methods from \cite{fem-4th-order,BEM-dirichlet}. To numerically approximate $\tau_1$ we use separation of variables  since $D$ is the unit circle. Therefore, we have that the `modified' Dirichlet eigenvalues for a constant $n$ satisfies 
$$ J_{|m|} \big( \tau \sqrt{n }\big)=0 \quad \text{ for all } \quad m \in \Z$$
and the `modified' plate buckling eigenvalues for a constant $n$ satisfies 
$$ \tau \sqrt{n}J_{|m|}' \big(\tau\sqrt{n}\big) -|m| J_{|m|}\big(\tau\sqrt{n}\big)=0 \quad \text{ for all } \quad m \in \Z. $$

To determine the approximate refractive index we first find the polynomial interpolation for $\tau_1(n)$ for constant $n\in [2,7]$ via the `\texttt{polyfit}' command in \texttt{MATLAB}. Then we solve for the constant $n_{\text{approx}}$ such that 
$$ k_1(n,\eta) = \tau_1(n_{\text{approx}}) $$ 
via the `\texttt{fzero}' command in \texttt{MATLAB}. Since $\tau_1$ is a deceasing function of $n$ the above equation has a unique solution $n_{\text{approx}}$. The results are reported in Tables \ref{estimate-n1} and \ref{estimate-n2} for the spherically stratified refractive indices used in our previous calculations with variable coefficient conductivity parameters. 

\begin{table}[ht!]
\centering  
\begin{tabular}{| c | c | c |} 
\hline                  
refractive index      &  1st  eigenvalue  &  $n_{\text{approx}}$    \\ [0.5ex] 
\hline                  
$n=4  $ & $ \, 2.00296851019\, $&  $\, 3.65301\, $  \\
$n=4+\exp(-r^2) $& $1.83137577764$ &  $4.39393$ \\
$n=4*\mathbbm{1}_{(0.25 \leq r < 1)} + 2*\mathbbm{1}_{(r<0.25)}   $ & $ 2.21052625727 $ &  $3.00333$ \\
\hline 
\end{tabular}
\caption{Estimation of the refractive index $n$ for $\eta = 1/\big(10+\sin(2\theta)^2\big)$. }\label{estimate-n1}
\end{table}

\begin{table}[ht!]
\centering  
\begin{tabular}{| c | c | c |} 
\hline                  
refractive index      &  1st eigenvalue  &  $n_{\text{approx}}$    \\ [0.5ex] 
\hline                  
$n=4  $ & $\, 1.25192566197 \, $&  $\, 3.866691\, $  \\
$n=4+\exp(-r^2) $& $ 1.11323689887$ &  $4.681572$ \\
$n=4*\mathbbm{1}_{(0.25 \leq r < 1)} + 2*\mathbbm{1}_{(r<0.25)}   $ & $1.30146479659 $ &  $3.404278$ \\
\hline 
\end{tabular}
\caption{Estimation of the refractive index $n$ for $\eta = 25(2+\sin^4(\theta))$. }\label{estimate-n2}
\end{table}

Simple calculus  gives that the average value of $n=4+\exp(-r^2)$ in the unit disk to be $4+(1-\exp(-1)) \approx 4.6321205$ which is fairly close to the approximation in Table \ref{estimate-n2}. We can also compute the average value for the piece-wise constant  refractive index $n=4*\mathbbm{1}_{(0.25 \leq r < 1)} + 2*\mathbbm{1}_{(r<0.25)} $ in the unit disk which is $3.875$. Also notice that due to the monotonicity of $\tau_1$ we have that $n_{\text{min}} \leq n_{\text{approx}}\leq n_{\text{max}}$. In the case of the classical transmission eigenvalues it has been numerically documented that estimating the refractive index by a constant leads to determining it's average value \cite{isp-n2}. Table \ref{estimate-n2}  seems to suggest that for $\eta \gg 1$ that estimating the refractive index by a constant may also lead to determining the average value.

\subsection{{\color{black}A Numerical Example for the Unit Square}}\label{square}
For completeness we provide numerical examples for the unit square $D=(0,1)^2$. This is given to show that this method also works for polygonal domains with no reentrant corners. Here we wish to show the accuracy of the approximation for this domain. To this end, we will compute the zero-index transmission eigenvalues for a constant $\eta$. To establish that the approximation is accurately computing the eigenvalues we will show test convergence as $\eta \to \infty$ as well as the monotonicity. We will also estimate the refractive index $n(x_1 ,x_2)$ assuming $\eta \gg 1$ just as we did for the unit sphere. Therefore, the zero-index transmission eigenvalues $k$ should converge to the `modified' Dirichlet eigenvalues for the unit square.  

For the approximate we have that  the basis functions are taken to be
$$\phi_{j}(x_1,x_2) = \sin(p\pi x_1)\sin(q\pi x_2) \quad \text{with index} \quad j=j(p,q) \in \N.$$
In the numerical examples we take 25 basis functions where $1\leq p,q\leq 5$ which gives the spectral approximation space  as 
$$ \text{Span} \Big\{ \phi_{j(p,q)}(x_1,x_2)  \Big\}_{p, q=1}^{5} \subset X(D).$$
To compute the zero-index transmission eigenvalues we proceed just as in the previous section. The spectral approximation of the eigenvalues $k_N$ satisfy the corresponding matrix eigenvalue problem with the appropriate mass and stiffness matrices. Again the matrix eigenvalue problem is solved by using `\texttt{polyeig}' command in \texttt{MATLAB}. 
\begin{table}[ht!]
\centering  
\begin{tabular}{| c | c | c | c | c |} 
\hline                  
     & $\eta=1$    & $\eta=10$  & $\eta=100$ & $\eta=1000$  \\ [0.5ex] 
\hline                  
$\, k_1\, $ & $2.42379135332$ & $2.23942914304$ &  $2.22322075768$ & $2.22161920572$ \\
\hline 
\end{tabular}
\caption{The first zero-index transmission eigenvalue for $n=4$ and $\eta=10^m$ with $m=0,1,2,3$ of the unit square. Here we see that $k_1$ is converging to the first `modified' Dirichlet eigenvalues for the unit square $\pi/\sqrt{2} \approx  2.2214414$.}\label{square1}
\end{table}

\begin{table}[ht!]
\centering  
\begin{tabular}{| c | c | c | c | c |} 
\hline                  
     & $\eta=1$    & $\eta=10$  & $\eta=100$ & $\eta=1000$  \\ [0.5ex] 
\hline                  
$\, k_1\, $ & $2.23128387981$ & $2.00960194537$ &  $1.96995925736$ & $1.96292112038$ \\
\hline 
\end{tabular}
\caption{The first zero-index transmission eigenvalue for $n= \left(\frac{x_1^2}{2}+2 \right) \hspace{-0.2cm}\left(\frac{x_2^2}{2}+2 \right)$ and $\eta=10^m$ with $m=0,1,2,3$ of the unit square.}\label{square2}
\end{table}

Here we see in Table \ref{square1} the convergence of the first zero-index transmission eigenvalue $k_1$ to the first  `modified' Dirichlet eigenvalue. Also, notice that in Tables \ref{square1} and \ref{square2} the monotonicity of the transmission eigenvalue $k_1(n,\eta)$ with respect to $n$ and $\eta$ is verified by the calculations. Now from the approximated transmission eigenvalue we can again estimate the refractive index. Using the convergence as $\eta \to \infty$ we proceed just as in the previous section. That is we find $n_{\text{approx}}$ such that 
$$k_1(n,\eta) = \tau_1(n_{\text{approx}})$$ 
where the conductivity parameter $\eta \gg 1$. To estimate the refractive index we can solve the above equation exactly since the `modified' Dirichlet eigenvalues are known analytically. In Tables \ref{square3} we present the approximation of two refractive indices from the first zero-index transmission eigenvalue for the unit square.

\begin{table}[ht!]
\centering  
\begin{tabular}{| c | c | c |} 
\hline                  
refractive index      &  1st eigenvalue  &  $n_{\text{approx}}$    \\ [0.5ex] 
\hline                  
$n=4  $ & $\, 2.23942914304 \, $&  $\, 3.935999\, $  \\
$n= \left(\frac{x_1^2}{2}+2 \right) \hspace{-0.2cm} \left(\frac{x_2^2}{2}+2 \right)$ & $ \, 2.00960194537 \,$ &  $\, 4.887757\, $ \\
\hline 
\end{tabular}
\caption{Estimation of the refractive index $n$ for $\eta = 10$. }\label{square3}
\end{table}

\section{Summary and Conclusions}
In conclusion, we have provided a numerical method to compute the zero-index transmission eigenvalues via the Dirichlet spectral-Galerkin approximation method. Our approximation space is taken to be the span of the first $N$ Dirichlet eigenfunctions. Even though our numerical examples are only presented in for the unit disk in $\R^2$ the analysis is also valid for any  domain {\color{black} where the boundary $\partial D$ is either polygonal with no reentrant corners or class $\mathscr{C}^2$} in $\R^d$ for $d=2$, 3. In order to apply this method one needs the Dirichlet eigenpairs a prior for the domain of interest. In theory this can be done by pre-calculating a fixed number of Dirichlet eigenpairs for the domain $D$ via BEM \cite{BEM-dirichlet} or FEM \cite{eig-FEM-book}. We have also given numerical examples to validate the theoretical results as well as investigated estimating the refractive index from the first zero-index transmission eigenvalue. It seems that for $\eta$ sufficiently large one can estimate the average value of $n$ which can be used for nondestructive testing. Possible future work would consist of applying this method to compute `classical' transmission eigenvalues with a conductive boundary and considering the inverse problem of recovering a variable coefficient refractive index from the eigenvalues.


\end{document}